\documentclass[12pt]{article}
\usepackage{amsfonts,amsmath,amsthm, amssymb}
\usepackage{latexsym, euscript, epic, eepic}
\newtheorem{defn}{Definition}[section]
\newtheorem{prop}{Proposition}[section]
\newtheorem{thm}{Theorem}[section]
\newtheorem{lem}{Lemma}[section]
\newtheorem{rem}{\bf Remark}[section]

\begin{document}
\title{The variational principle of topological pressure for actions of
sofic groupoids
 }
\author{Xiaoyao Zhou ,$^\dag$  Ercai Chen$^{\dag \ddag}$ \\
\small \it $\dag$ School of Mathematical Sciences and Institute of Mathematics, Nanjing Normal University,\\
\small \it Nanjing 210023, P.R.China,
\small \it $\ddag$ Center of Nonlinear Science,\\
\small \it Nanjing University, Nanjing 210093,  P.R.China.\\
\small \it e-mail:
\small \it\noindent zhouxiaoyaodeyouxian@126.com\\
\small \it \noindent ecchen@njnu.edu.cn}
\date{}
\maketitle

\begin{center}
\begin{minipage}{120mm}
{\small {\bf Abstract.} This article establishes the variational
principle of topological pressure for actions of  sofic groupoids. }
\end{minipage}
\end{center}

\vskip0.5cm
\section{Introduction}
Variational principles are beautiful results in a dynamical system.
Establishing variational principle is an important topic in
dynamical system theory. The first variational principle that
reveals the relationship between topological entropy and
measure-theoretic entropy was obtained by L. Goodwyn \cite{[Goo2]}
and T. Goodman \cite{[Goo1]}. M. Misiurewicz gave a short proof in
\cite{[Mis]}. R. Bowen \cite{[Bow0]} established the variational
principle of entropy for non-compact set in 1973. Y. Pesin and B.
Pitskel \cite{[PesPit]} studied the variational principle of
pressure for non-compact set.

Romagnoli \cite{[Rom]} introduced two types of measure-theoretic
entropies relative to a finite open cover and proved the variational
principle between local entropy and one type. Later, Glasner and
Weiss \cite{[GlaWei]} proved that if the system is invertible, then
the local variational principle is also true for another type
measure-theoretic entropy.  W. Huang \& Y. Yi \cite{[HuaYi]}
obtained the variational principle of local pressure. W. Huang, X.
Ye and G. Zhang \cite{[HuaYeZha]} established the variational
principle of local entropy for a countable discrete amenable group
action. B. Liang and K. Yan \cite{[LiaYan]} generalized it to local
pressure for sub-additive potentials of amenable group actions.

Recently, L. Bowen \cite{Bow2}, \cite{Bow1} introduced a notion of
entropy for measure-preserving actions of a countable discrete sofic
group on a standard probability space admitting a generating finite
partition. Just after that, D. Kerr and H. Li \cite{[KerLi1]}
 established variational principle of entropy for
sofic group actions. G. Zhang \cite{[Zha]} generalized it to the
variational principle of local entropy for countable infinite sofic
group actions. N. Chung \cite{[Chu]} generalized the result of D.
Kerr and H. Li \cite{[KerLi1]} to the variational principle of
pressure for sofic group actions.

Very recently, L. Bowen \cite{Bow3} generalized the sofic entropy
theory to class-bijective extensions of sofic groupoids, and
established the variational principle between topological entropy
and measure entropy.

This article establishes the variational principle of topological
pressure for actions of  sofic groupoids.

\section{Preliminaries}
The section presents some basic definitions and notations about
groupoid.
\subsection{Discrete groupoid}

A groupoid $\mathcal{H}^1$ is a set consists of morphisms, equipped
with a set of objects $\mathcal{H}^0,$ source and range maps
$\varsigma,\tau:\mathcal{H}^1\to\mathcal{H}^0,$ an injective
inclusion map $i:\mathcal{H}^0\to\mathcal{H}^1,$ such that

(i) for any $x\in \mathcal{H}^0,\varsigma(i(x))=\tau(i(x))=x;$

(ii) let $\mathcal{H}^2=\left\{(f,g)\in\mathcal{H}^1\times
\mathcal{H}^1:\varsigma(f)=\tau(g)\right\}$ and define a composition
map $c:\mathcal{H}^2\to\mathcal{H}^1$ satisfying

(iii) for any
$(f,g)\in\mathcal{H}^2,\varsigma(c(f,g))=\varsigma(g),\tau(c(f,g))=\tau(f);$

(iv)for any $f\in\mathcal{H}^1,$ there exists a unique element
$f^{-1}\in\mathcal{H}^1$ with $c(f^{-1},f)=i(\varsigma(f))$ and
$c(f,f^{-1})=i(\tau(f)).$

For convenience, a remark about the notations of a groupoid to be
used in this article is presented here.

\begin{rem}{\rm Let $\mathcal{H}^1$ be a groupoid defined as above.
\begin{itemize}
  \item
Write $\mathcal{H}$ instead of $\mathcal{H}^1, fg$ instead of
$c(fg)$ for the sake of convenience.

\item Let $[[\mathcal{H}]]$ be the collection of all Borel subsets
$f\subset\mathcal{H},$ whose source and range maps restricted to $f$
are Borel isomorphisms onto their respective images.

\item  Given
$f,g\in[[\mathcal{H}]],$ set $f^{-1}=\{h^{-1}:h\in f\}$ and
$fg=\{h\in\mathcal{H}:h=f'g'{\rm ~for~some~}f'\in f,g'\in g\}.$

\item $\mathcal{H}\in[[\mathcal{H}]]$ and any Borel subset $P$ of
$\mathcal{H}^0$ belongs to $[[\mathcal{H}]]$ (Using the inclusion
map, $\mathcal{H}^0$ can be viewed as a subset of $\mathcal{H}.$).

\item If
$f\in[[\mathcal{H}]]$ and $x\in\varsigma(f),$ then $fx$ is
well-defined, and $fx=\varsigma^{-1}(x)\cap f.$ However,  $fx$ may
not be contained in  $\mathcal{H}^0.$ To avoid this, we let $f\cdot
x=\tau(fx)\in \mathcal{H}^0.$ In a similar way, if $P\subset
\mathcal{H}^0,$ then we let $f\cdot P=\tau(fP).$

\item Let $[\mathcal{H}]\subset[[\mathcal{H}]]$ be the collection of all
Borel subsets $f\subset\mathcal{H},$ whose source and range maps are
each Borel isomorphisms onto $\mathcal{H}^0.$
 Composition makes
$[[\mathcal{H}]]$ an inverse semi-group called the semi-group of
partial automorphisms, $[\mathcal{H}]$ a group called the full group
of $\mathcal{H}$.

\item Let $[\mathcal{H}]_{top}$ be the collection of all closed  sets
$f\subset\mathcal{H}$ such that the restrictions of the source and
range maps to $f$ are homeomorphisms onto $\mathcal{H}^0.$ Under
composition, $[\mathcal{H}]_{top}$  is a subgroup of
$[\mathcal{H}].$

\end{itemize}

}\end{rem}

A measurable groupoid  $\mathcal{H}$ means  a groupoid $\mathcal{H}$
is equipped with the structure of a standard Borel space satisfying
that $\mathcal{H}^0$ is a Borel set and the source, range,
composition and inversion maps are all Borel.

A groupoid $\mathcal{H}$ is discrete means that for any
$x\in\mathcal{H}^0,\varsigma^{-1}(x)$ and $\tau^{-1}(x)$ are
countable.

A discrete probability measure groupoid refers to a discrete
measurable groupoid $\mathcal{H}$ together with a Borel probability
measure $\nu$ over $\mathcal{H}^0$ satisfying that if
$\nu_\varsigma,\nu_\tau$ are the measures over $\mathcal{H}$ defined
by

\begin{equation*}
\nu_{\varsigma}(B)=\int_{\mathcal{H}^0}|\varsigma^{-1}(x)\cap
B|d\nu(x), \nu_{\tau}(B)=\int_{\mathcal{H}^0}|\tau^{-1}(x)\cap
B|d\nu(x),
\end{equation*}
for any Borel set $B\subset \mathcal{H},$ then $\nu_{\varsigma}$ is
equivalent to $\nu_{\tau}.$ In particular, if
$\nu_{\varsigma}=\nu_{\tau},$ then $(\mathcal{H},\nu)$ is pmp
(probability-measure-preserving). The article deals with pmp
groupoids. Hence, we let $\nu$ denote $\nu_{\varsigma}=\nu_{\tau},$
and $\nu$ restricted to $\mathcal{H}^0$ is $\nu.$

A discrete topological groupoid is a discrete groupoid $\mathcal{H}$
paired with a topology such that the structure maps (source, range,
inverse and composition) are continuous. An open subset $f\subset
\mathcal{H}$ is called a bisection means that the restrictions of
the source and range maps to $f$ are homeomorphisms onto their
images which are open subsets of $\mathcal{H}^0.$ $\mathcal{H}$ is
\'{e}tale refers to that if  every  $g\in\mathcal{H}$ is contained
in a bisection.

\begin{rem}{\rm Assume that $(\mathcal{H},\nu)$ is pmp.

\begin{itemize}
\item
Given  a Borel set $A\subset\mathcal{H}^0,$ set $\partial A=A
\cap\overline{\mathcal{H}^0\setminus A}.$

\item Let $\mathcal{B}_{\partial}(\mathcal{H}^0,\nu)$ be the set of
all Borel subset $A\subset\mathcal{H}^0$ satisfying $\nu(\partial
A)=\nu(\partial(\mathcal{H}^0\setminus A))=0.$

\item Let $[[\mathcal{H}]]_{top}$ be the collection of all elements of
$[[\mathcal{H}]]$ with the form $f=\cup_{i=1}^n f_i,$ where

(i) for each $i$ there is a bisection $U_i$ with $f_i\subset U_i,$

(ii)
$\{\varsigma(f_i)\}_{i=1}^n\subset\mathcal{B}_{\partial}(\mathcal{H}^0,\nu)$
are pairwise disjoint,

(iii)
$\{\tau(f_i)\}_{i=1}^n\subset\mathcal{B}_{\partial}(\mathcal{H}^0,\nu)$
are pairwise disjoint.

\item
For $f\in[[\mathcal{H}]],$ the trace of $f$ is given by
$tr_{\mathcal{H}}(f)=\nu(\mathcal{H}^0\cap f).$ Also
$|f|_{\mathcal{H}}=\nu(f).$
\end{itemize}

}\end{rem}

{\bf Example \cite{Bow3}} Let $d\in\mathbb{N}.$ The full groupoid on
$\{1,\cdots,d\}$ is $\Delta_d:=\{1,\cdots,d\}^2$ and the unit space
is $\Delta_d^0:=\{(i,i):1\leq i\leq d\}.$ The structure maps are
given by $\varsigma(i,j)=(j,j),\tau(i,j)=(i,i),(i,j)^{-1}=(j,i)$ and
$(i,j)(j,k)=(i,k).$ Let $\zeta_d(E)=|E|/ d$ for any set $E\subset
\Delta_d. (\Delta_d,\zeta_d)$ is a pmp groupoid. We simplify
$[\Delta_d]$ as $[d]$ that is isomorphic with the symmetric group on
$\{1,\cdots,d\}$ and simplify $[[\Delta_d]]$ as $[[d]]$ that is the
set of all subsets $f\subset\Delta_d$ such that
$\varsigma:f\to\varsigma(f)$ and $\tau:f\to\tau(f)$ are bijections.
Let $tr_d:=tr_{\Delta_d},|\cdot|_d:=|\cdot|_{\Delta_d}$ and
$tr_d(f)=|f\cap\Delta_d^0|/d$ and $|f|_d=|f|/d.$
\subsection{Sofic approximations, actions, spanning and separating sets}
Let $(\mathcal{H},\nu)$ denote a pmp discrete groupoid. Given $d>0,$
let Map$([[\mathcal{H}]],[[d]])$ be the collection of all Borel maps
from $[[\mathcal{H}]]$ to $[[d]].$ For a finite set $F\subset
[[\mathcal{H}]]$ and $\sigma:[[\mathcal{H}]]\to[[d]],$ set

\begin{equation*}
N(\sigma,F)=\left\{\sigma'\in{\rm
Map}([[\mathcal{H}]],[[d]]):\sigma'(f)=\sigma(f),\forall f\in
F\right\}.
\end{equation*}
The Borel structure of Map$([[\mathcal{H}]],[[d]])$ can be viewed as
generated by all such $N(\sigma,F).$

$X\subset_f Y$ means that $X$ is a finite subset of $Y.$ Let
$F\subset_f [[\mathcal{H}]],\delta>0,$ a map
$\sigma:[[\mathcal{H}]]\to [[d]]$ is $(F,\delta)-$multiplicative
means for any $s,t \in F,|\sigma(st)\Delta
\sigma(s)\sigma(t)|_d<\delta$ and $(F,\delta)$-trace-preserving
refers to $|tr_d(\sigma(s))-tr_{\mathcal{H}}(s)|<\delta$ for any
$s\in F.$

\begin{defn}{\rm\cite{Bow3}}(Sofic approximation)
Let $J$ be a direct set. For each $j\in J,$ let $d_j\in\mathbb{N}$
and $\mathbb{P}_j$ be a Borel probability measure on {\rm
Map}$([[\mathcal{H}]],[[d_j]]).$ The family
$\mathbb{P}=\{\mathbb{P}_j\}_{j\in J}$ is a sofic approximation to
$(\mathcal{H},\nu)$ means

(i) for any $F\subset_f[[\mathcal{H}]]$ and $\delta>0,$

\begin{equation*}
\lim\limits_{j\to J}\mathbb{P}(\{\sigma\in{\rm
Map}([[\mathcal{H}]],[[d_j]]):\sigma{\rm~is~}(F,\delta)-{\rm
trace-preserving}\})=1.
\end{equation*}

(ii) for any $F\subset_f[[\mathcal{H}]],\delta>0,$ there exists
$j\in J$ such that $j'\geq j$ implies $\mathbb{P}_{j'}$-almost every
$\sigma$ is $(F,\delta)$-multiplicative.

(iii) $\lim\limits_{j\to J}d_j=+\infty.$
\end{defn}

The groupoid $(\mathcal{H},\nu)$ is sofic means it has a sofic
approximation.

Let $\mathcal{G},\mathcal{H}$ be measurable groupoid. A map
$\pi:\mathcal{G}\to\mathcal{H}$ is a groupoid morphism means that
for every  $(f,g)\in\mathcal{G}^2,\pi(fg)=\pi(f)\pi(g),$ for each
$f\in\mathcal{G},\pi(f)^{-1}=\pi(f^{-1})$ and
$\pi(\mathcal{G}^0)\subset\mathcal{H}^0.$ It is class-bijective if
for every $a\in\mathcal{G}^0,$ the restriction of $\pi$ to
$\varsigma^{-1}(a)$ is a bijection onto $\varsigma^{-1}(\pi(a))$ and
the restriction of $\pi$ to $\tau^{-1}(a)$ is also a bijection onto
$\tau^{-1}(\pi(a)).$ If $\pi$ is also surjective, then $\mathcal{G}$
is a class-bijective extension of $\mathcal{H}$ or, equivalently,
$\mathcal{H}$ is a class-bijective factor of $\mathcal{G}.$ $\pi$ is
pmp (probability-measure-preserving) if $\pi_\star\mu=\nu$ and
$(\mathcal{G},\mu),(\mathcal{H},\nu)$ are pmp groupoid.

If $\pi:(\mathcal{G},\mu)\to(\mathcal{H},\nu)$ is class-bijective
then $\pi^{-1}:[[\mathcal{H}]]\to[[\mathcal{G}]]$ is a
homeomorphism. Assume $\mathcal{G}$ and $\mathcal{H}$ are
topological groupoid, $\pi$ is continuous and pmp then
$\pi^{-1}([[\mathcal{H}]]_{top})\subset[[\mathcal{G}]]_{top}.$ Given
$x\in\mathcal{G}$ and $f\in[[\mathcal{H}]], fx:=\pi^{-1}(f)x.$ If
$x\in\mathcal{G}^0$ then we let $f\cdot x$ be
$\tau(\pi^{-1}(f)x)=\tau(fx).$

\begin{defn}
Given a pseudo-metric space $(Z,\rho)$ and $\epsilon>0,$ a subset
$Y\subset Z$ is $(\rho,\epsilon)$-separated means that for any
$y_1\neq y_2\in Y, \rho(y_1,y_2)>\epsilon.$ Given $X\subset Z,$ let
$N_\epsilon(X,\rho)$ be the maximum cardinality of a
$(\rho,\epsilon)$-separated subset $Y\subset X.$ For $X,Y\subset
Z,Y(\rho,\epsilon)$-spans $X$ means that for each $x\in X,$ there
exists $y\in Y$ with $\rho(x,y)<\epsilon.$ Let $N'_\epsilon(X,\rho)$
be the minimum cardinality of a set $Y\subset Z$ which
$(\rho,\epsilon)$-spans $X.$
\end{defn}

\section{Topological pressure}
Assume $\mathcal{G},\mathcal{H}$ are two discrete separable
topological groupoid such that $\mathcal{G}^0$ and $\mathcal{H}^0$
are compact metrizable  spaces, $\pi:\mathcal{G}\to\mathcal{H}$ is a
continuous class-bijective factor map, $(\mathcal{H},\nu)$ is a pmp
groupoid. Given a sofic approximation
$\mathbb{P}=\{\mathbb{P}_j\}_{j\in J}$ to $(\mathcal{H},\nu),
\varphi\in C(\mathcal{G}^0,\mathbb{R}),$ a bias $\beta$ and $p\in
[1,\infty],$ we define the sofic topological pressure of $\pi$ with
respect to $(\mathbb{P},p,\beta,\varphi),$ via a choice of
generating pseudo-metric.

For an integer $d>0,$ we write $x\in(\mathcal{G}^0)^d$ as
$x=(x_1,\cdots, x_d).$ For $f\in [[\mathcal{H}]],$ we set $f\cdot
x:=\{f\cdot x_1,\cdots,f\cdot x_d\}. f\cdot x_i$ is not defined when
$\pi(x_i)\notin \varsigma(f).$ When this occurs, we set $f\cdot
x_i:=\star,$ where $\star$ is a special symbol. So $f\cdot
x\in(\mathcal{G}^0\cup\{\star\})^d.$

$\Delta_d^0$ is viewed as $\{1,\cdots,d\}$ and given
$\sigma:[[\mathcal{H}]]\to[[d]], f\in[[\mathcal{H}]]$ and
$i\in\{1,\cdots,d\},$ we write $\sigma(f)i\in\{1,\cdots,d\}$ instead
of $\sigma(f)\cdot i.$ With $x$ as above, we define
$x\circ\sigma(f):=(x_{\sigma(f)1},\cdots,x_{\sigma(f)d}).$ If
$i\notin \varsigma(\sigma(f)),$ then $x_{\sigma(f)i}$ is not
well-defined. When it occurs, we let $x_{\sigma(f)i}:=\star.$ Thus
$x\circ \sigma(f)\in (\mathcal{G}^0\cup\{\star\})^d.$

Choose a continuous pseudo-metric $\rho$ on $\mathcal{G}^0. \rho$
can be extended to $\mathcal{G}^0\cup\{\star\}$ by setting
$\rho(\star,\star)=0$ and
$\rho(\star,x)=\max\{\rho(y,z):y,z\in\mathcal{G}^0\}$ for any
$x\in\mathcal{G}^0.$ Two pseudo-metrics on
$(\mathcal{G}^0\cup\{\star\})^d$ are given by

\begin{equation*}
\rho_2(x,x'):=\left(\frac{1}{d}\sum\limits_{i=1}^d\rho(x_i,x'_i)^2\right)^{1/2},
\rho_\infty(x,x'):=\max\limits_{1\leq i\leq d}\rho(x_i,x'_i).
\end{equation*}

\begin{defn}{\rm\cite{Bow3}}
(Approximate partial orbits) Let $C(\mathcal{H}^0)$ be the space of
continuous complex-valued functions on $\mathcal{H}^0.$ For a map
$\sigma:[[\mathcal{H}]]\to[[d]],$ finite sets $F\subset
[[\mathcal{H}]], K\subset C(\mathcal{H}^0)$ and $\delta>0,$ set

\begin{equation*}\begin{split}
&Orb_\nu(\pi,\sigma,F,K,\delta,\rho):=\{(x_1,\cdots,x_d)\in(\mathcal{G}^0)^d:\delta>\rho_2(f\cdot
x,x\circ\sigma(f)),\forall f\in
F,\\&\delta>|d^{-1}\sum\limits_{i=1}^dk(\pi(x_i))-\int k
d\nu|,\forall k\in K\}.
\end{split}\end{equation*}
\end{defn}

\begin{defn}{\rm\cite{Bow3}}
A bias $\beta$ for $J$ is either an element of $\{-,+\}$ or a
nonprincipal ultrafilter on $J.$ Given a function
$\Phi:J\to\mathbb{R},$ if $\beta$ is an ultrafilter then the
ultralimit $\lim\limits_{j\to\beta}\Phi(j)$ is well-defined.
Otherwise, define

\begin{equation*}
\lim\limits_{j\to\beta}\Phi(j):=\left\{
                                 \begin{array}{ll}
                                   \liminf_{j\to J}\Phi(j), & if \beta=-; \\
                                   \limsup_{j\to J}\Phi(j), & if \beta=+.
                                 \end{array}
                               \right.
\end{equation*}
\end{defn}

Let
\begin{equation*}
M_\epsilon(Orb_\nu(\pi,\sigma,F,K,\delta,\rho),\rho_2,\varphi):=\sup\limits_{\mathcal{E}}\left\{\sum\limits_{(x_1,\cdots,x_d)\in\mathcal{E}}
\exp(\sum\limits_{i=1}^d\varphi( x_i))\right\},
\end{equation*}
where $\mathcal{E}$ runs over $(\rho_2,\epsilon)$-separated subsets
of $Orb_\nu(\pi,\sigma,F,K,\delta,\rho).$

If $1\leq p<\infty,$ then we define
\begin{equation*}
\|M_\epsilon(Orb_\nu(\pi,\cdot,F,K,\delta,\rho),\rho_2,\varphi)\|_{p,\mathbb{P}_j}=\left(\int
M_\epsilon(Orb_\nu(\pi,\sigma,F,K,\delta,\rho),\rho_2,\varphi)^pd\mathbb{P}_j(\sigma)\right)^{1/p}.
\end{equation*}
Define topological pressure by separated set as follows:
\begin{equation*}\begin{split}
&P^\beta_{\mathbb{P},p}(\pi,\rho,2,\varphi):=\sup\limits_{\epsilon>0}\inf\limits_{\delta>0}\inf\limits_{F\subset_f[[\mathcal{H}]]_{top}}\inf\limits_{K\subset_f
C(\mathcal{H}^0)}\lim\limits_{j\to\beta}\frac{1}{d_j}\log
\|M_\epsilon(Orb_\nu(\pi,\cdot,F,K,\delta,\rho),\rho_2,\varphi)\|_{p,\mathbb{P}_j},\\&
P^\beta_{\mathbb{P},p}(\pi,\rho,\infty,\varphi):=\sup\limits_{\epsilon>0}\inf\limits_{\delta>0}\inf\limits_{F\subset_f[[\mathcal{H}]]_{top}}\inf\limits_{K\subset_f
C(\mathcal{H}^0)}\lim\limits_{j\to\beta}\frac{1}{d_j}\log
\|M_\epsilon(Orb_\nu(\pi,\cdot,F,K,\delta,\rho),\rho_\infty,\varphi)\|_{p,\mathbb{P}_j}.
\end{split}
\end{equation*}

\begin{rem}
We are suppressing the choice of bias $\beta$ and parameter
$p\in[1,\infty]$ from the notation. Thus,
$P^\beta_{\mathbb{P},p}(\pi,\rho,2,\varphi)$ can be simplified as
$P_{\mathbb{P}}(\pi,\rho,2,\varphi)$ and
$P^\beta_{\mathbb{P},p}(\pi,\rho,\infty,\varphi)$ can be expressed
by $P_{\mathbb{P}}(\pi,\rho,\infty,\varphi).$
\end{rem}

\begin{rem}\label{rem3.2}
The quantity
$\frac{1}{d_j}\log\|M_\epsilon(Orb_\nu(\pi,\cdot,F,K,\delta,\rho),\rho_2,\varphi)\|_{p,\mathbb{P}_j}$
is monotone increasing  in $\delta$ and monotone decreasing in
$\epsilon, F, K.$ This implies that the supremum and the infimums
can be replaced by the appropriate limits.
\end{rem}

\noindent Let
\begin{equation*}
M'_\epsilon(Orb_\nu(\pi,\sigma,F,K,\delta,\rho),\rho_2,\varphi):=\inf\limits_{\mathcal{E}'}\left\{\sum\limits_{(x_1,\cdots,x_d)\in\mathcal{E}'}
\exp(\sum\limits_{i=1}^d\varphi( x_i))\right\},
\end{equation*}
where $\mathcal{E}'$ runs over $(\rho_2,\epsilon)$-spanning subsets
of $Orb_\nu(\pi,\sigma,F,K,\delta,\rho).$

\begin{lem}
Topological pressure defined by spanning set is equal to given by
separated set. More precisely,

\begin{equation*}\begin{split}
&P_{\mathbb{P}}(\pi,\rho,2,\varphi)=\sup\limits_{\epsilon>0}\inf\limits_{\delta>0}\inf\limits_{F\subset_f[[\mathcal{H}]]_{top}}\inf\limits_{K\subset_f
C(\mathcal{H}^0)}\lim\limits_{j\to\beta}\frac{1}{d_j}\log
\|M'_\epsilon(Orb_\nu(\pi,\cdot,F,K,\delta,\rho),\rho_2,\varphi)\|_{p,\mathbb{P}_j},(*)\\&
P_{\mathbb{P}}(\pi,\rho,\infty,\varphi)=\sup\limits_{\epsilon>0}\inf\limits_{\delta>0}\inf\limits_{F\subset_f[[\mathcal{H}]]_{top}}\inf\limits_{K\subset_f
C(\mathcal{H}^0)}\lim\limits_{j\to\beta}\frac{1}{d_j}\log
\|M'_\epsilon(Orb_\nu(\pi,\cdot,F,K,\delta,\rho),\rho_\infty,\varphi)\|_{p,\mathbb{P}_j}(**).
\end{split}\end{equation*}
\end{lem}

\begin{proof}
The proof of  $(*)$ is presented here. Since $(**)$ is  similar to
$(*),$ the proof is omitted.

Let $\mathcal{E}\subset Orb_\nu(\pi,\sigma,F,K,\delta,\rho)$ be a
 $(\rho_2,\epsilon)$-separated subset with maximal
cardinality, then $\mathcal{E}(\rho_2,2\epsilon)$-spans
$Orb_\nu(\pi,\sigma,F,K,\delta,\rho).$ Therefore,

\begin{equation*}\begin{split}
&M_\epsilon(Orb_\nu(\pi,\sigma,F,K,\delta,\rho),\rho_2,\varphi)\\&\geq
\sum\limits_{(x_1,\cdots,x_d)\in\mathcal{E}}\exp(\sum\limits_{i=1}^d\varphi(x_i))\\&
\geq M'_{2\epsilon}
(Orb_\nu(\pi,\sigma,F,K,\delta,\rho),\rho_2,\varphi).
\end{split}\end{equation*}
This implies that

\begin{equation*}
P_{\mathbb{P}}(\pi,\rho,2,\varphi)\geq\sup\limits_{\epsilon>0}\inf\limits_{\delta>0}\inf\limits_{F\subset_f[[\mathcal{H}]]_{top}}\inf\limits_{K\subset_f
C(\mathcal{H}^0)}\lim\limits_{j\to\beta}\frac{1}{d_j}\log
\|M'_\epsilon(Orb_\nu(\pi,\cdot,F,K,\delta,\rho),\rho_2,\varphi)\|_{p,\mathbb{P}_j}.
\end{equation*}
Next, the proof of the  opposite inequality is presented as follows.
For any $x',y'\in\mathcal{G}^0,$ there exists $\eta_1>0$ such that $
|\varphi( x')-\varphi( y')|<\eta_1,$ whenever
$\rho(x',y')<\sqrt{\epsilon}.$ There exists
$(\rho_2,\epsilon)$-separated subset $\mathcal{E}_1\subset
Orb_\nu(\pi,\sigma,F,K,\delta,\rho)$ such that

\begin{equation*}
M_\epsilon(Orb_\nu(\pi,\sigma,F,K,\delta,\rho),\rho_2,\varphi)\leq
\sum\limits_{(x^1_1,\cdots,x_d^1)\in
\mathcal{E}_1}\exp(\sum\limits_{i=1}^d\varphi( x_i^1))\cdot \exp(1)
\end{equation*}
and there  exists a $(\rho,\epsilon/2)$-spanning subset
$\mathcal{E}_2$ of $Orb_\nu(\pi,\sigma,F,K,\delta,\rho)$ such that

\begin{equation*}
M'_{\epsilon/2}(Orb_\nu(\pi,\sigma,F,K,\delta,\rho),\rho_2,\varphi)\cdot
\exp(1)\geq \sum\limits_{(x^2_1,\cdots,x_d^2)\in
\mathcal{E}_2}\exp(\sum\limits_{i=1}^d\varphi( x_i^2)).
\end{equation*}
Define $\phi:\mathcal{E}_1\to\mathcal{E}_2$ by choosing, for each
$(x_1^1,\cdots,x_d^1)\in\mathcal{E}_1,$ some point
$\phi((x_1^1,\cdots,x_d^1))=(x_1^2,\cdots,x_d^2)$ with
$\rho_2((x_1^1,\cdots,x_d^1),(x_1^2,\cdots,x_d^2))<\epsilon/2.$ Note
that $\phi$ is injective and
$\rho_2((x_1^1,\cdots,x_d^1),(x_1^2,\cdots,x_d^2))<\epsilon/2$
implies that $|\{1\leq i\leq
d:\rho(x_i^1,x_i^2)\geq\sqrt{\epsilon}\}|\leq\lceil\frac{d\epsilon}{4}\rceil.$
Therefore,

\begin{equation*}\begin{split}
&M'_{\epsilon/2}(Orb_\nu(\pi,\sigma,F,K,\delta,\rho),\rho_2,\varphi)\cdot
\exp(1)\\&\geq \sum\limits_{(x^2_1,\cdots,x_d^2)\in
\mathcal{E}_2}\exp(\sum\limits_{i=1}^d\varphi( x_i^2))\\&\geq
\sum\limits_{(x^2_1,\cdots,x_d^2)\in
\phi\mathcal{E}_1}\exp(\sum\limits_{i=1}^d\varphi( x_i^2))\\&=
\sum\limits_{(x^1_1,\cdots,x_d^1)\in
\mathcal{E}_1\atop(x^2_1,\cdots,x_d^2)= \phi((x^1_1,\cdots,x^1_d))
}\exp(\sum\limits_{i=1}^d\varphi( x_i^1)+\sum\limits_{i=1}^d\varphi(
x_i^2)-\sum\limits_{i=1}^d\varphi( x_i^1))\\&\geq
\sum\limits_{(x^1_1,\cdots,x_d^1)\in
\mathcal{E}_1}\exp(\sum\limits_{i=1}^d\varphi(
x_i^1))\exp(-(d-\lceil\frac{d\epsilon}{4}\rceil)\eta_1-2\|\varphi\|\lceil\frac{d\epsilon}{4}\rceil)\\&\geq
M_\epsilon(Orb_\nu(\pi,\sigma,F,K,\delta,\rho),\rho_2,\varphi)\exp(-(d-\lceil\frac{d\epsilon}{4}\rceil)\eta_1-2\|\varphi\|\lceil\frac{d\epsilon}{4}\rceil-1).
\end{split}\end{equation*}
This implies

\begin{equation*}
P_{\mathbb{P}}(\pi,\rho,2,\varphi)\leq\sup\limits_{\epsilon>0}\inf\limits_{\delta>0}\inf\limits_{F\subset_f[[\mathcal{H}]]_{top}}\inf\limits_{K\subset_f
C(\mathcal{H}^0)}\lim\limits_{j\to\beta}\frac{1}{d_j}\log
\|M'_\epsilon(Orb_\nu(\pi,\cdot,F,K,\delta,\rho),\rho_2,\varphi)\|_{p,\mathbb{P}_j}.
\end{equation*}
\end{proof}

\begin{rem}
In the definitions of $
M_\epsilon(Orb_\nu(\pi,\sigma,F,K,\delta,\rho),\rho_2,\varphi)$ it
suffices to take the supremum over those $(\rho_2,\epsilon)$
separated sets with maximal cardinality. Similarly, in the
definitions of $
M'_\epsilon(Orb_\nu(\pi,\sigma,F,K,\delta,\rho),\rho_2,\varphi),$ it
suffices to take the infimum over the $(\rho_2,\epsilon)$ spanning
sets with minimal cardinality.
\end{rem}

\begin{lem}
$P_\mathbb{P}(\pi,\rho,2,\varphi)=P_\mathbb{P}(\pi,\rho,\infty,\varphi).$
\end{lem}

\begin{proof}
Since any $(\rho_2,\epsilon)$-separated set is also
$(\rho_\infty,\epsilon)$-separated, we get

\begin{equation*}
M_\epsilon(Orb_\nu(\pi,\sigma,F,K,\delta,\rho),\rho_2,\varphi)\leq
M_\epsilon(Orb_\nu(\pi,\sigma,F,K,\delta,\rho),\rho_\infty,\varphi),
\end{equation*}
for any $\sigma,F,K,\delta.$ This implies that
$P_{\mathbb{P}}(\pi,\rho,2,\varphi)\leq
P_{\mathbb{P}}(\pi,\rho,\infty,\varphi).$

Next, we prove the opposite direction. Let $\frac{1}{10}>\epsilon>0,
\Omega$ be a $(\rho,\sqrt{\epsilon})$-spanning subset of
$\mathcal{G}^0$ with minimum cardinality. Let $\delta>0,
F\subset_f[[\mathcal{H}]],K\subset_f
C(\mathcal{H}^0),\sigma:[[\mathcal{H}]]\to[[d]]$ and
$\mathcal{E}\subset(\mathcal{G}^0)^d$ be a
$(\rho_2,\epsilon)$-spanning set for
$Orb_\nu(\pi,\sigma,F,K,\delta,\rho)$ with

\begin{equation*}
2M'_\epsilon(Orb_\nu(\pi,\sigma,F,K,\delta,\rho),\rho_2,\varphi)\geq\sum\limits_{(y_1,\cdots,y_d)\in\mathcal{E}}\exp(\sum\limits_{i=1}^d\varphi(y_i)).
\end{equation*}
Let $\eta=\lceil\epsilon d\rceil.$ Define
$\mathcal{E}'\subset(\mathcal{G}^0)^d$ as follows:

For $y=(y_1,\cdots,y_d)\in\mathcal{E},$ every set
$\Lambda\subset[d]$ of cardinality $\eta$ and every map
$\phi:\Lambda\to\Omega, y^\phi$ is given by

\begin{equation*}
y_i^\phi=\left\{
           \begin{array}{ll}
             y_i, & i\notin\Lambda; \\
             \phi(i), & i\in\Lambda.
           \end{array}
         \right.
\end{equation*}
Let $\mathcal{E}'$ be the collection of all $y^\phi$ over all such
$y\in\mathcal{E}$ and $\phi:\Lambda\to \Omega.$ Then

\begin{equation*}
\sum\limits_{y^\phi\in\mathcal{E}'}\exp(\sum\limits_{i=1}^d\varphi(x_i))\leq\binom{d}{\eta}|\Omega|^\eta\sum\limits_{y\in\mathcal{E}}\exp(
\sum\limits_{i=1}^d\varphi(y_i)+2\|\varphi\|\eta).
\end{equation*}
We claim that $\mathcal{E}'$ is
$(\rho_\infty,\sqrt{\epsilon})$-spanning for
$Orb_\nu(\pi,\sigma,F,K,\delta,\rho).$ In fact, let $z\in
Orb_\nu(\pi,\sigma,F,K,\delta,\rho).$ Since $\mathcal{E}$ is
$(\rho_2,\epsilon)$-spanning for
$Orb_\nu(\pi,\sigma,F,K,\delta,\rho),$ there exists
$y\in\mathcal{E}$ such that $\rho_2(y,z)\leq\epsilon,$ that is,

\begin{equation*}
\frac{1}{d}\sum\limits_{i=1}^d\rho(y_i,z_i)^2\leq \epsilon^2.
\end{equation*}
So, there exists a set $\Lambda\subset [d]$ such that for
$i\notin\Lambda,\rho(y_i,z_i)\leq\sqrt{\epsilon}$ and
$|\Lambda|=\eta.$ By the definition of $\Omega,$ for every $i\in
\Lambda,$ there exists a point $\phi(i)\in \Omega$ such that
$\rho(\phi(i),z_i)\leq \sqrt{\epsilon}.$ Thus,
$\rho_\infty(y^\phi,z)\leq\sqrt{\epsilon},$ that shows that
$\mathcal{E}'$  is $(\rho_\infty,\sqrt{\epsilon})$-spanning for
$Orb_\nu(\pi,\sigma,F,K,\delta,\rho).$ It follows that

\begin{equation*}\begin{split}
&M_\epsilon'(Orb_\nu(\pi,\sigma,F,K,\delta,\rho),\rho_\infty,\varphi)\\&\leq\sum\limits_{y^\phi\in\mathcal{E}'}
\exp(\sum\limits_{i=1}^d\varphi(y_i^\phi))\\&\leq\binom{d}{\eta}|\Omega|^\eta\sum\limits_{y\in\mathcal{E}}\exp(
\sum\limits_{i=1}^d\varphi(y_i)+2\|\varphi\|\eta)\\&
\leq2\exp(2\|\varphi\|\eta)\binom{d}{\eta}|\Omega|^\eta\cdot
M'_\epsilon(Orb_\nu(\pi,\sigma,F,K,\delta,\rho),\rho_2,\varphi).
\end{split}\end{equation*}
Then the desired result follows from Stirling's approximation
formula.
\end{proof}

\begin{defn}{\rm\cite{Bow3}}
A pseudo-metric $\rho$ on $\mathcal{G}^0$ is dynamically generating
for $\pi:\mathcal{G}\to\mathcal{H}$ means that for each
$x,y\in\mathcal{G}^0$ there exists $f\in [\mathcal{H}]_{top}$ with
$\rho(f\cdot x,f\cdot y)>0.$
\end{defn}

\begin{thm}\label{thm3.1}
Let $\rho, \rho'$ be two dynamically generating continuous
pseudo-metrics on $\mathcal{G}^0.$ Then
$P_\mathbb{P}(\pi,\rho,\varphi)=P_\mathbb{P}(\pi,\rho',\varphi).$
\end{thm}

\begin{rem}
The proof of Theorem \ref{thm3.1} uses only properties (ii) and
(iii) of the definition of sofic approximation.
\end{rem}

\begin{defn}
Given Theorem \ref{thm3.1}, the sofic topological pressure of $\pi$
with respect to $(\mathbb{P},p,\beta,\varphi)$ is given by
$P_\mathbb{P}(\pi,\varphi):=P_\mathbb{P}(\pi,\rho,\varphi)=P^\beta_{\mathbb{P},p}(\pi,\varphi),$
where $\rho$ is any dynamically generating continuous pseudo-metric
on $\mathcal{G}^0.$
\end{defn}
It is worth mentioning that this is the relative pressure with
respect to the measure $\nu.$ Because the sofic approximation
$\mathbb{P}$ determines $\nu,\nu$ is implicitly referenced in the
notation.

\begin{lem}\label{lem3.3}
If $\rho,\rho'$ are two continuous metrics on $\mathcal{G}^0,$ then
$P_\mathbb{P}(\pi,\rho,\varphi)=P_\mathbb{P}(\pi,\rho',\varphi).$
\end{lem}

\begin{proof}
Because $\rho$ and $\rho'$ are continuous metrics on
$\mathcal{G}^0,$ and $\mathcal{G}^0$ is compact, for every
$\delta>0$ and sufficiently large integer $n>>0,$ there exist
$\delta_0,\epsilon_n>0$ such that
\begin{itemize}
  \item $\rho'(x,y)\leq\sqrt{\delta_0}\Rightarrow\rho(x,y)\leq\delta,$
  \item $\delta_0\leq\delta^2;$
  \item $\rho'(x,y)\geq\epsilon_n\Rightarrow\rho(x,y)\geq\frac{1}{n};$
  \item $\lim\limits_{n\to\infty}\epsilon_n=0.$
\end{itemize}
Let $\Omega=\max\{\rho(x,y):x,y\in\mathcal{G}^0\}$ be the diameter
of $\rho.$ We claim that for any
$\sigma:[[\mathcal{H}]]\to[[d]],F\subset_f[[\mathcal{H}]]_{top}$ and
$K\subset_f C(\mathcal{H}^0),$

\begin{equation*}
Orb_\nu(\pi,\sigma,F,K,\delta_0,\rho')\subset
Orb_\nu(\pi,\sigma,F,K,\delta(\Omega^2+1)^{1/2},\rho),
\end{equation*}
the proof of which is presented in Lemma 6.9 of \cite{Bow3}. By
choice of $\epsilon_n,$

\begin{equation*}
M_{\epsilon_n}(Orb_\nu(\pi,\sigma,F,K,\delta_0,\rho'),\rho'_\infty,\varphi)\leq
M_{\frac{1}{n}}(Orb_\nu(\pi,\sigma,F,K,\delta(\Omega^2+1)^{1/2},\rho),\rho_\infty,\varphi).
\end{equation*}
Thus,

\begin{equation*}\begin{split}
&\lim_{j\to\beta}\frac{1}{d_j}\log\|M_{\epsilon_n}(Orb_\nu(\pi,\cdot,F,K,\delta_0,\rho'),\rho'_\infty,\varphi)\|_{p,\mathbb{P}_j}\\&
\leq\lim_{j\to\beta}\frac{1}{d_j}\log\|M_{\frac{1}{n}}(Orb_\nu(\pi,\cdot,F,K,\delta(\Omega^2+1)^{1/2},\rho),\rho_\infty,\varphi)\|_{p,\mathbb{P}_j}.
\end{split}\end{equation*}
Taking  the infimum over $\delta_0>0$ then over all $\delta>0,$ then
over all $F\subset_f[[\mathcal{H}]]_{top}$ and $K\subset_f
C(\mathcal{H}^0),$ then the supremum over all $n,$ we have
$P_\mathbb{P}(\pi,\rho',\varphi)\leq
P_\mathbb{P}(\pi,\rho,\varphi).$ The lemma is implied by the
arbitrariness of $\rho$ and $\rho'.$
\end{proof}

For a pseudo-metric $\rho$ on $\mathcal{G}^0,$ and a sequence
$\{\phi_i\}_{i=1}^\infty$ with $\phi_i\in[[\mathcal{H}]]_{top},$ a
pseudo-metric $\rho^\phi$ on $\mathcal{G}^0$ is given by

\begin{equation*}
\rho^\phi(x,y)=\left(\sum\limits_{i=1}^\infty2^{-i}\rho(\phi_i\cdot
x,\phi_i\cdot y)^2\right)^{1/2}.
\end{equation*}
If $m$ is a continuous metric on $\mathcal{G}^0,$ then a metric
$\overline{m}$ on $[\mathcal{H}]_{top}$ can be given by
$\overline{m}(f,g)=\sup\limits_{x\in\mathcal{H}^0}m(f\cdot x,g\cdot
x).$ And the topology of $[\mathcal{H}]_{top}$ is independent of the
choice of metric $m.$

\begin{lem}\label{lem3.4}{\rm\cite{Bow3}}
Considering the homeomorphism group of $\mathcal{G}^0,
Homeo(\mathcal{G}^0),$ with the topology of pointwise convergence,
$[\mathcal{H}]_{top}$ is equipped with the topology such that it
inherits as a subgroup of $Homeo(\mathcal{G}^0).$ if $\rho$ is a
continuous dynamically generating pseudo-metric and
$\{\phi_i\}_{i=1}^\infty$ is a dense subset of
$[\mathcal{H}]_{top},$ then $\rho^\phi$ is a continuous metric.
\end{lem}

\begin{lem}\label{lem3.5}
If $\rho$ is a continuous dynamically generating pseudo-metric and
$\{\phi_i\}_{i=1}^\infty\subset[\mathcal{H}]_{top}$ with
$\phi_1=\mathcal{H}^0,$ then

\begin{equation*}
P_\mathbb{P}(\pi,\rho,\varphi)=P_\mathbb{P}(\pi,\rho^\phi,\varphi).
\end{equation*}
\end{lem}

\begin{proof}
For any $x,y\in(\mathcal{G}^0)^d, $ it follows from a
straightforward computation that $
\rho_2^\phi(x,y)^2=\sum\limits_{j=1}^\infty 2^{-j}\rho_2(\phi_j\cdot
x,\phi_j\cdot y)^2.$ Since $\phi_1=\mathcal{H}^0, \rho\leq
2\rho^\phi,$ we have

\begin{equation*}
Orb_\nu(\pi,\sigma,F,K,\delta,\rho^\phi)\subset
Orb_\nu(\pi,\sigma,F,K,2\delta,\rho)
\end{equation*}
for any $\sigma,F,K,\delta.$ We claim that

{\bf Claim:} There exists a $(\rho_2,2\epsilon)$-spanning set
$\mathcal{E}$ for $Orb_\nu(\pi,\sigma,F,K,2\delta,\rho),$ which is
contained in $Orb_\nu(\pi,\sigma,F,K,2\delta,\rho),$ and satisfies

\begin{equation*}\begin{split}
&\sum\limits_{(x_1,\cdots,
x_d)\in\mathcal{E}}\exp(\sum\limits_{i=1}^d\varphi(x_i))\\&\leq
\exp(2\|\varphi\|\lceil\epsilon d\rceil+(d-\lceil\epsilon
d\rceil)\eta_1+1)\cdot
M'_\epsilon(Orb_\nu(\pi,\sigma,F,K,2\delta,\rho),\rho_2,\varphi),
\end{split}\end{equation*}
where $\eta_1>0$ such that $ |\varphi( x')-\varphi( y')|<\eta_1,$
whenever  $x',y'\in\mathcal{G}^0, $ with
$\rho(x',y')<\sqrt{\epsilon}.$

\begin{proof}
Let $\mathcal{E}'$ be a minimal $(\rho_2,\epsilon)$-spanning set for
$Orb_\nu(\pi,\sigma,F,K,2\delta,\rho),$ with

\begin{equation*}
\sum\limits_{(x'_1,\cdots,x'_d)\in\mathcal{E}'}\exp(\sum\limits_{i=1}^d\varphi(x'_i))\leq
M'_\epsilon(Orb_\nu(\pi,\sigma,F,K,2\delta,\rho),\rho_2,\varphi)\cdot
\exp(1).
\end{equation*}
For each $x'=(x'_1,\cdots,x'_d)\in\mathcal{E}',$ there is an element
$x=(x_1,\cdots,x_d)\in Orb_\nu(\pi,\sigma,F,K,2\delta,\rho)$ such
that $\rho_2(x,x')<\epsilon.$ Let $\mathcal{E}$ be the collection of
such $x,$ then $\mathcal{E}$ is a $(\rho_2,2\epsilon)$-spanning set
for $Orb_\nu(\pi,\sigma,F,K,2\delta,\rho).$  Because
$\rho_2(x,x')<\epsilon,$ there exists $\Lambda\subset [d]$ such that
$|\Lambda|=\lceil\epsilon d\rceil$ and
$\rho(x_i,x'_i)<\sqrt{\epsilon}$ for $i\in[d]\setminus \Lambda.$
Therefore,

\begin{equation*}\begin{split}
&M'_\epsilon(Orb_\nu(\pi,\sigma,F,K,2\delta,\rho),\rho_2,\varphi)\cdot
\exp(1)\\&\geq\sum\limits_{(x'_1,\cdots,x'_d)\in\mathcal{E}'}\exp(\sum\limits_{i=1}^d\varphi(x'_i))\\&\geq
\sum\limits_{(x_1,\cdots,
x_d)\in\mathcal{E}}\exp(\sum\limits_{i=1}^d\varphi(x_i)-2\|\varphi\|\lceil\epsilon
d\rceil-(d-\lceil\epsilon d\rceil)\eta_1).
\end{split}\end{equation*}
The result is desired.
\end{proof}

Let $\Omega$ be the diameter of $(\mathcal{G}^0,\rho), F_n$ be any
finite subset of $[[\mathcal{H}]]_{top}$ containing
$\{\phi_1,\cdots,\phi_n\}.$ If $1\leq i\leq n$ and $x,y\in
Orb_\nu(\pi,\sigma,F_n,K,2\delta,\rho)$ satisfy
$\rho_2(x,y)<2\epsilon,$ then

\begin{equation*}\begin{split}
&\rho_2(\phi_i\cdot x,\phi_i\cdot y)\\&\leq
\rho_2(x\circ\sigma(\phi_i),y\circ\sigma(\phi_i))+
\rho_2(x\circ\sigma(\phi_i),\phi_i\cdot
x)+\rho_2(y\circ\sigma(\phi_i),\phi_i\cdot y)\\&<2\epsilon+4\delta.
\end{split}\end{equation*}
This implies that
$\rho_2^\phi(x,y)^2=\sum\limits_{j=1}^\infty2^{-j}\rho_2(\phi_j\cdot
x,\phi_j\cdot y)^2<2^{-n}\Omega^2+(2\epsilon+4\delta)^2.$ By Claim,
there exists a $(\rho_2,2\epsilon)$-spanning set $\mathcal{E}$ for
$Orb_\nu(\pi,\sigma,F_n,K,2\delta,\rho),$ which is contained in
$Orb_\nu(\pi,\sigma,F_n,K,2\delta,\rho),$ and satisfies

\begin{equation*}\begin{split}
&\sum\limits_{(x_1,\cdots,
x_d)\in\mathcal{E}}\exp(\sum\limits_{i=1}^d\varphi(x_i))\\&\leq
\exp(2\|\varphi\|\lceil\epsilon d\rceil+(d-\lceil\epsilon
d\rceil)\eta_1+1)\cdot
M'_\epsilon(Orb_\nu(\pi,\sigma,F_n,K,2\delta,\rho),\rho_2,\varphi).
\end{split}\end{equation*}
So, for any $x\in Orb_\nu(\pi,\sigma,F_n,K,2\delta,\rho)$ there
exists $y\in\mathcal{E}$ with $\rho_2(x,y)<2\epsilon,$ which implies
$\rho^\phi_2(x,y)^2<2^{-n}\Omega^2+(2\epsilon+4\delta)^2.$ Thus,
$\mathcal{E}$ is
$(\rho_2^\phi,\sqrt{2^{-n}\Omega^2+(2\epsilon+4\delta)^2})$-spanning.
Let $\eta=\sqrt{2^{-n}\Omega^2+(2\epsilon+4\delta)^2}.$ Then

\begin{equation*}\begin{split}
&\exp(2\|\varphi\|\lceil\epsilon d\rceil+(d-\lceil\epsilon
d\rceil)\eta_1+1)\cdot
M'_\epsilon(Orb_\nu(\pi,\sigma,F_n,K,2\delta,\rho),\rho_2,\varphi)\\&\geq
M'_\eta(Orb_\nu(\pi,\sigma,F_n,K,2\delta,\rho),\rho^\phi_2,\varphi)\\&\geq
M'_\eta(Orb_\nu(\pi,\sigma,F_n,K,\delta,\rho^\phi),\rho^\phi_2,\varphi),
\end{split}\end{equation*}
where the last inequality is due to the inclusion
$Orb_\nu(\pi,\sigma,F_n,K,\delta,\rho^\phi)\subset
Orb_\nu(\pi,\sigma,F_n,K,2\delta,\rho).$ By monotonicity, if $n$ is
large enough and $\delta$ is small enough then $3\epsilon>\eta$
which implies

\begin{equation*}\begin{split}
&\exp(2\|\varphi\|\lceil\epsilon d\rceil+(d-\lceil\epsilon
d\rceil)\eta_1+1)\cdot
M'_\epsilon(Orb_\nu(\pi,\sigma,F_n,K,2\delta,\rho),\rho_2,\varphi)\\&\geq
M'_{3\epsilon}(Orb_\nu(\pi,\sigma,F_n,K,\delta,\rho^\phi),\rho^\phi_2,\varphi).
\end{split}\end{equation*}
Since $F_n$ is any finite subset of $[[\mathcal{H}]]_{top}$
containing $\{\phi_1,\cdots, \phi_n\} $ and Remark \ref{rem3.2}, we
have

\begin{equation*}\begin{split}
&2\|\varphi\|\epsilon+(1-\epsilon)\eta_1+\inf\limits_{\delta>0}\inf\limits_{F\subset_f[[\mathcal{H}]]_{top}}\inf\limits_{K\subset_f
C(\mathcal{H}^0)}\lim\limits_{j\to\beta}\frac{1}{d_j}\log\|M'_\epsilon(Orb_\nu(\pi,\cdot,F,K,2\delta,\rho),\rho_2,\varphi)\|_{p,\mathbb{P}_j}\\&\geq
\inf\limits_{\delta>0}\inf\limits_{F\subset_f[[\mathcal{H}]]_{top}}\inf\limits_{K\subset_f
C(\mathcal{H}^0)}\lim\limits_{j\to\beta}\frac{1}{d_j}\log\|M'_{3\epsilon}(Orb_\nu(\pi,\cdot,F,K,\delta,\rho^\phi),\rho^\phi_2,\varphi)\|_{p,\mathbb{P}_j}.
\end{split}\end{equation*}
Since this is true for every $\epsilon>0$ and $\eta_1\to0$ as
$\epsilon\to0,$ we have $P_\mathbb{P}(\pi,\rho,2,\varphi)\geq
P_\mathbb{P}(\pi,\rho^\phi,2,\varphi),$ which implies
$P_\mathbb{P}(\pi,\rho,\varphi)\geq
P_\mathbb{P}(\pi,\rho^\phi,\varphi).$

Now, it's turn to show the opposite inequality. It was claimed that
given any finite $F\subset[[\mathcal{H}]]_{top}$ with
$\mathcal{H}^0\in F$ and $\delta>0,$ if $n$ is sufficiently large,
$F'=\{\phi_jf:f\in F,1\leq j\leq n\}$ and $\sigma$ is
$(F',\delta^2/M^2)$-multiplicative then

\begin{equation*}
Orb_\nu(\pi,\sigma,F,K,2\delta,\rho^\phi)\supset
Orb_\nu(\pi,\sigma,F',K,\delta,\rho)
\end{equation*}
in \cite{Bow3}. The claim and $\rho\leq 2\rho^\phi$  show that

\begin{equation*}
M'_\epsilon(Orb_\nu(\pi,\sigma,F,K,2\delta,\rho^\phi),\rho^\phi_\infty,\varphi)\geq
M'_{2\epsilon}(Orb_\nu(\pi,\sigma,F',K,\delta,\rho),\rho_\infty,\varphi).
\end{equation*}
This implies that $P_\mathbb{P}(\pi,\rho^\phi,\varphi)\geq
P_\mathbb{P}(\pi,\rho,\varphi).$
\end{proof}

{\it Proof of Theorem \ref{thm3.1}}  By \cite{Bow3}, there exists a
dense sequence
$\phi=\{\phi_i\}_{i=1}^\infty\subset[\mathcal{H}]_{top}$ with
$\phi_1=\mathcal{H}^0.$ By Lemmas \ref{lem3.5}, \ref{lem3.4},
\ref{lem3.3},
$P_\mathbb{P}(\pi,\rho,\varphi)=P_\mathbb{P}(\pi,\rho^\phi,\varphi)=P_\mathbb{P}(\pi,\rho'^{\phi},\varphi)=P_\mathbb{P}(\pi,\rho',\varphi).$
\hfill$\Box$
\section{Measure entropy}
There are two definitions of measure sofic entropy for groupoid. One
is defined via partitions by Bowen \cite{Bow3} in a manner analogous
to \cite{Ker}, the other is via pseudo-metric like topological
entropy.

Let $\pi:(\mathcal{G},\mu)\to(\mathcal{H},\nu)$ be a class-bijective
extension of pmp discrete groupoids and
$\mathbb{P}:=\{\mathbb{P}_j\}_{j\in J}$ a sofic approximation to
$(\mathcal{H},\nu).$ For a finite partition $\mathcal{P}$ of
$\mathcal{G}^0$  and $F\subset_f[[\mathcal{H}]],$ let
$\mathcal{P}^F$ be the coarsest partition of $\mathcal{G}^0$
containing $\{f\cdot P:f\in F,P\in\mathcal{P}\}.$ Let
$\Sigma(\mathcal{P})$ denote the smallest sigma algebra of
$\mathcal{G}^0$ containing $\mathcal{P}.$ Let
$\mathcal{B}(\Delta_d^0)$ be the collection of all subsets of
$\Delta_d^0.$ A map
$\phi:\Sigma(\mathcal{P})\to\mathcal{B}(\Delta_d^0)$ is a
homeomorphism means that given $A,B\in\Sigma(\mathcal{P}),\phi(A\cup
B)=\phi(A)\cup\phi(B),\phi(A\cap B)=\phi(A)\cap\phi(B)$ and
$\phi(\mathcal{G}^0)=\Delta_d^0.$

\begin{defn}{\rm\cite{Bow3}(Good homomorphism)}
For $\sigma:[[\mathcal{H}]]\to[[d]],f\in[[\mathcal{H}]],\delta>0$
and $F\subset[[\mathcal{H}]],$ let
$Hom(\pi,\sigma,\mathcal{P},F,\delta)$ be the collection of all
homomorphisms $\phi:\Sigma(\mathcal{P}^F)\to\mathcal{B}(\Delta_d^0)$
such that

(1)
$\sum\limits_{P\in\mathcal{P}}|\sigma(f)\cdot\phi(P)\Delta\phi(f\cdot
P)|<d\delta,\forall f\in F;$

(2)
$\sum\limits_{P\in\mathcal{P}^F}||\phi(P)|d^{-1}-\mu(P)|<\delta.$
\end{defn}

Given a partition $\mathcal{Q}$ of $\mathcal{G}^0$ with
$\mathcal{Q}\leq\mathcal{P},$ let
$|Hom(\pi,\sigma,\mathcal{P},F,\delta)|_\mathcal{Q}$ be the
cardinality of the set of homomorphisms
$\phi:\Sigma(\mathcal{Q})\to\mathcal{B}(\Delta_d^0)$ such that there
exists a $\phi'\in Hom(\pi,\sigma,\mathcal{P},F,\delta)$ satisfying
that $\phi$ is the restriction of $\phi'$ to $\Sigma(\mathcal{Q}).$

\begin{defn}{\rm\cite{Bow3}}\label{defn4.2}
Let $\mathcal{B}(\mathcal{G}^0)$ be the Borel sigma algebra of
$\mathcal{G}^0.$ For finite Borel partitions
$\mathcal{Q}\leq\mathcal{P}$ and a sub-algebra
$\mathcal{F}\subset\mathcal{B}(\mathcal{G}^0),$ define

\begin{equation*}\begin{split}
h_{\mathbb{P},\mu}(\pi,\mathcal{Q},\mathcal{P},F,\delta)&:=\lim\limits_{j\to\beta}\frac{1}{d_j}\log\|
|Hom(\pi,\cdot,\mathcal{P},F,\delta)|_\mathcal{Q}\|_{p,\mathbb{P}_j}\\
h_{\mathbb{P},\mu}(\pi,\mathcal{Q},\mathcal{P})&:=\inf\limits_{F\subset_f[[\mathcal{H}]]}\inf\limits_{\delta>0}h_{\mathbb{P},\mu}(\pi,\mathcal{Q},\mathcal{P},F,\delta)\\
h_{\mathbb{P},\mu}(\pi,\mathcal{Q},\mathcal{F})&:=\inf\limits_{\mathcal{Q}\leq\mathcal{P}\subset\mathcal{F}}h_{\mathbb{P},\mu}(\pi,\mathcal{Q},\mathcal{P})\\
h_{\mathbb{P},\mu}(\pi,\mathcal{F})&:=\sup\limits_{\mathcal{Q}\subset\mathcal{F}}h_{\mathbb{P},\mu}(\pi,\mathcal{Q},\mathcal{F}).
\end{split}\end{equation*}
The infimum in the second-to-last line runs over all finite Borel
partitions $\mathcal{P}$ with $\mathcal{Q}\leq
\mathcal{P}\subset\mathcal{F}$ and the supremum in the last line
runs over all finite partition $\mathcal{Q}\subset\mathcal{F}.$ The
sofic measure entropy of $\pi$ with respect to $\mathbb{P},p,\beta$
is given by
$h_{\mathbb{P},\mu}(\pi):=h_{\mathbb{P},\mu}(\pi,\mathcal{B}(\mathcal{G}^0)).$
\end{defn}

\begin{rem}{\rm\cite{Bow3}}
\begin{itemize}
  \item $h_{\mathbb{P},\mu}(\pi)$ depends implicity on $1\leq p\leq
\infty$ and a bias $\beta.$
  \item The order of the supremums, infimums and limits above is
important with the exception that the three infimums can be permuted
without affecting the definition of
$h_{\mathbb{P},\mu}(\pi,\mathcal{Q},\mathcal{F}).$
  \item The quantity $\frac{1}{d_j}\log \||Hom(\pi,\cdot,\mathcal{P},F,\delta)|_\mathcal{Q}\|_{p,\mathbb{P}_j}$
\end{itemize}
is monotone increasing in $\delta,\mathcal{Q}$ and decreasing in
$F,\mathcal{P}.$
\end{rem}

\begin{defn}{\rm\cite{Bow3}}
Let $(\mathcal{H},\nu)$ be a discrete pmp groupoid,
$\sigma:[[\mathcal{H}]]\to[[d]],F\subset_f[[\mathcal{H}]]$ and
$\delta>0. \sigma$ is $(F,\delta)$-continuous means that
$|\sigma(f)\Delta\sigma(g)|_d<\delta+\nu(f\Delta g),\forall f,g\in
F.$ If $\mathbb{P}=\{\mathbb{P}_j\}_{j\in J}$ is a sofic
approximation to $(\mathcal{H},\nu),$ then $\mathbb{P}$ is
asymptotically continuous  means that for every
$F\subset_f[[\mathcal{H}]],\delta>0,$ there exists $j\in J$ such
that $j'\geq j$ implies $\mathbb{P}_{j'}$-almost every $\sigma$ is
$(F,\delta)$-continuous.
\end{defn}

Now, let's present the formulation of measure entropy via
pseudo-metrics.

The assumption is given as follows: two discrete pmp topological
groupoids $\mathcal{G},\mathcal{H}$ such that $\mathcal{G}^0$ and
$\mathcal{H}^0$ are compact metrizable spaces, a class-bijective
continuous factor $\pi:(\mathcal{G},\mu)\to (\mathcal{H},\nu),$ a
sofic approximation $\mathbb{P}=\{\mathbb{P}_j\}_{j\in J}$ to
$(\mathcal{H},\nu),$ a continuous pseudo-metric $\rho$ on
$\mathcal{G}^0,$ a bias $\beta$ and $p\in [1,\infty].$ From this
data, the sofic measure entropy of $(\pi,\rho)$ with respect to
$(\mathbb{P},p,\beta)$ can be defined. Furthermore, if $\rho$ is
dynamically generating, $\mathcal{H}$ is \'etale, $\nu$ is regular
and $\mathbb{P}$ is asymptotically continuous, then this entropy
coincides with the Definition \ref{defn4.2}. The reader is referred
to \cite{Bow3} for details.

\begin{defn}{\rm\cite{Bow3}}
For a map
$\sigma:[[\mathcal{H}]]\to[[d]],F\subset_f[[\mathcal{H}]]_{top},K\subset_f
C(\mathcal{G}^0)$ and $\delta>0,$ let
$Orb_\mu(\pi,\sigma,F,K,\delta,\rho)$ be the set of all
$(x_1,\cdots,x_d)\in Orb_\nu(\pi,\sigma,F,\emptyset,\delta,\rho)$
such that

\begin{equation*}
\max\limits_{x\in
K}\left|\frac{1}{d}\sum\limits_{i=1}^dk(x_i)-\int_{\mathcal{G}^0}k
d\mu\right|<\delta.
\end{equation*}
Define

\begin{equation*}\begin{split}
&h_{\mathbb{P},\mu}(\pi,\rho,2):=\sup\limits_{\epsilon>0}\inf\limits_{\delta>0}\inf\limits_{F\subset_f[[\mathcal{H}]]_{top}}\inf\limits_{K\subset_f
C(\mathcal{G}^0)}\lim\limits_{j\to\beta}\frac{1}{d_j}\log
\|N_\epsilon(Orb_\mu(\pi,\cdot,F,K,\delta,\rho),\rho_2)\|_{p,\mathbb{P}_j},\\&
h_{\mathbb{P},\mu}(\pi,\rho,\infty):=\sup\limits_{\epsilon>0}\inf\limits_{\delta>0}\inf\limits_{F\subset_f[[\mathcal{H}]]_{top}}\inf\limits_{K\subset_f
C(\mathcal{G}^0)}\lim\limits_{j\to\beta}\frac{1}{d_j}\log
\|N_\epsilon(Orb_\mu(\pi,\cdot,F,K,\delta,\rho),\rho_\infty)\|_{p,\mathbb{P}_j}.
\end{split}
\end{equation*}

\begin{rem}{\rm \cite{Bow3}}
\begin{itemize}
  \item $h_{\mathbb{P},\mu}(\pi,\rho,2)$ and
$h_{\mathbb{P},\mu}(\pi,\rho,\infty)$ depend implicity on the choice
of bias $\beta$ and parameter $p\in[1,\infty].$
  \item The order of the supremums, infimums and limits above is
important with the exception that the three infimums can be permuted
without affecting the definition.
   \item The quantity
$\frac{1}{d_j}\log\|N_\epsilon(Orb_\mu(\pi,\cdot,F,K,\delta,\rho),\rho_2)\|_{p,\mathbb{P}_j}$
is monotone increasing in $\delta$ and decreasing in $\epsilon,F,K.$
\item Similar statements remain valid if $\rho_2$ is replaced with
$\rho_\infty$ or $N_\epsilon$ is replaced with $N'_\epsilon.$
 \item If we replace $N_\epsilon(\cdot)$ with $N'_\epsilon(\cdot),$
then we get equivalent definitions.
\item $h_{\mathbb{P},\mu}(\pi,\rho,2)=h_{\mathbb{P},\mu}(\pi,\rho,\infty).$
\item Let
$h_{\mathbb{P},\mu}(\pi,\rho):=h_{\mathbb{P},\mu}(\pi,\rho,2).$ If
$\rho,\rho'$ are dynamically generating continuous pseudo-metrics on
$\mathcal{G}^0$ then
$h_{\mathbb{P},\mu}(\pi,\rho)=h_{\mathbb{P},\mu}(\pi,\rho').$
\item If $\rho$ is a dynamically generating continuous pseudo-metric
on $\mathcal{G}^0,\mathcal{H}$ is \'etale, $\nu$ is regular and
$\mathbb{P}$ is asymptotically continuous then
$h_{\mathbb{P},\mu}(\pi,\rho)=h_{\mathbb{P},\mu}(\pi).$
\end{itemize}
\end{rem}

\end{defn}

\section{The variational principle}
This section is aim to reveal the relationship between topological
pressure and measure entropy. First, we present two propositions the
proof of which can be found in \cite{Bow3}.

\begin{prop}{\rm\cite{Bow3}}\label{prop5.1}
Let $(\mathcal{H},\nu)$ be a pmp separable \'{e}tale topological
discrete groupoid, $\mathcal{G}$ be a separable topological discrete
groupoid, and $\mu$ be a Borel probability measure on
$\mathcal{G}^0,$ and suppose that $\mu$ is
$[[\mathcal{H}]]_{top}$-invariant in the sense that $\mu(k\circ
f)=\mu(k\circ\tau(f))$ for any continuous function $k\in
C(\mathcal{G}^0)$ and $f\in[[\mathcal{H}]]_{top},$ then
$(\mathcal{G},\mu)$ is  pmp.
\end{prop}

\begin{prop}{\rm\cite{Bow3}}\label{prop5.2}
Assume that $\Omega$ is a directed set and
$\omega\in\Omega\mapsto\mu_\omega\in M(\mathcal{G}^0)$ is a map such
that $\lim\limits_{\omega\to\Omega}\mu_\omega=\mu_\infty,$ and
$\pi_*\mu_\infty=\nu.$ Then
$\lim\limits_{\omega\to\Omega}\mu_\omega(k\circ f)=\mu_\infty(k\circ
f),$ for any $k\in C(\mathcal{G}^0)$ and
$f\in[[\mathcal{H}]]_{top}.$
\end{prop}

\begin{thm}\label{thm5.1}
Let $(\mathcal{H},\nu)$ be a pmp separable \'{e}tale topological
discrete groupoid, $\mathcal{G}$ be a separable topological discrete
groupoid, $\pi:\mathcal{G}\to\mathcal{H}$ a continuous
class-bijective factor, and $\mathbb{P}=\{\mathbb{P}_j\}_{j\in J}$
an asymptotically continuous sofic approximation to
$(\mathcal{H},\nu).$ If $\mathcal{H}^0$ and $\mathcal{G}^0$ are
compact and metrizable and $\nu$ is regular, then for any
$\varphi\in C(\mathcal{G}^0,\mathbb{R}),p\in[1,\infty]$ and bias
$\beta\neq-,$

\begin{equation*}
P_\mathbb{P}(\pi,\varphi)=\max\limits_{\mu}h_{\mathbb{P},\mu}(\pi)+\int_{\mathcal{G}^0}\varphi
d\mu,
\end{equation*}
where the max runs over all measures $\mu$ on $\mathcal{G}^0$ such
that $\pi_*\mu=\nu$ and $(\mathcal{G},\mu)$ is pmp.
\end{thm}

\begin{proof}
The proof will be divided into the following two steps.

{\bf Step 1:} This step proves
$P_\mathbb{P}(\pi,\varphi)\leq\sup\limits_{\mu}h_{\mathbb{P},\mu}(\pi)+\int_{\mathcal{G}^0}\varphi
d\mu.$

Without loss of generality, we may assume
$P_\mathbb{P}(\pi,\varphi)>-\infty.$ Let $\kappa>0,$ then there
exists $\epsilon>0$ such that

\begin{equation*}
P^\epsilon_\mathbb{P}(\pi,\rho,2,\varphi)\geq
P_\mathbb{P}(\pi,\rho,2,\varphi)-\kappa,
\end{equation*}
where
$P^\epsilon_\mathbb{P}(\pi,\rho,2,\varphi):=\inf\limits_{\delta>0}\inf\limits_{F\subset_f[[\mathcal{H}]]_{top}}\inf\limits_{K\subset_f
C(\mathcal{H}^0)}\lim\limits_{j\to\beta}\frac{1}{d_j}\log
\|M_\epsilon(Orb_\nu(\pi,\cdot,F,K,\delta,\rho),\rho_2,\varphi)\|_{p,\mathbb{P}_j}.$
Let
$\Omega=\{(F,L,\delta):F\subset_f[[\mathcal{H}]]_{top},L\subset_f
C(\mathcal{G}^0),\delta>0\}.\Omega$ is considered as a directed set
by declaring $(F,L,\delta)\leq(F',L',\delta')$ if $F'\supset F,
L'\supset L,\delta'\leq \delta.$ Given
$\omega=(F_\omega,L_\omega,\delta_\omega)\in\Omega,$ let
$K_\omega=\{k\in C(\mathcal{H}^0):k\circ\pi\in L_\omega\}.$ Let
$M(\mathcal{G}^0)$ denote the space of Borel probability measures on
$\mathcal{G}^0.$

{\bf Claim:} There exists a directed net
$\omega\in\Omega\mapsto\mu_\omega\in M(\mathcal{G}^0)$ such that
\begin{description}
  \item[(i)] $h^\epsilon_{\mathbb{P},\mu_\omega}(\pi,\rho,\omega,2)+\int_{\mathcal{G}^0}\varphi d\mu_\omega+\delta_\omega
\geq P^\epsilon_\mathbb{P}(\pi,\rho,2,\varphi),$
where
\begin{equation*}\begin{split}
h^\epsilon_{\mathbb{P},\mu_\omega}(\pi,\rho,\omega,2)&:=h^\epsilon_{\mathbb{P},\mu_\omega}(\pi,\rho,F_\omega,L_\omega,\delta_\omega,2)\\&
:=\lim\limits_{j\to\beta}\frac{1}{d_j}\log\|N_\epsilon(Orb_{\mu_\omega}(\pi,\cdot,F_\omega,L_\omega,\delta_\omega,\rho),\rho_2)\|_{p,\mathbb{P}_j}.
\end{split}\end{equation*}

  \item[(ii)] $\lim\limits_{\omega\to\Omega}|\mu_\omega(k\circ f)-\mu_\omega(k\circ
\tau(f))|=0,$ for any $f\in[[\mathcal{H}]]_{top},k\in
C(\mathcal{G}^0).$
  \item[(iii)]
$\lim\limits_{\omega\to\Omega}|\mu_\omega(k\circ\pi)-\nu(k)|=0,$ for
any $k\in C(\mathcal{H}^0).$
\end{description}

\begin{proof}
For $\varphi\in L\subset_f C(\mathcal{G}^0)$ and $\delta>0,$ let
$M(L,\delta)$ be the set of all $\mu\in M(\mathcal{G}^0)$ such that
$|\mu(k\circ \pi)-\nu(k)|\leq \delta$ for any $k\in
C(\mathcal{H}^0)$ with $k\circ\pi\in L.$ For
$F\subset_f[[\mathcal{H}]]_{top},$ let $D(F,L,\delta)\subset
M(L,\delta)$ be a finite set satisfying that for every $\lambda\in
M(L,\delta)$ there exists a $\mu\in D(F,L,\delta)$ such that

\begin{equation*}
|\mu(k\circ f)-\lambda(k\circ f)|<\delta,~~\forall f\in F,k\in L.
\end{equation*}
For every $x\in(\mathcal{G}^0)^d,$ the measure $m_x\in
M(\mathcal{G}^0)$ is given by
$m_x=d^{-1}\sum\limits_{i=1}^d\delta_{x_i},$ where $\delta_{x_i}$ is
the Dirac measure concentrated on $x_i.$ For every
$\omega\in\Omega,$ choose a Borel map $x\in\{y\in(\mathcal{G}^0)^d:
m_y\in M(L_\omega,\delta_\omega)\}\mapsto \mu_{x,\omega}\in D(\omega)$ such that

\begin{equation*}
|\mu_{x,\omega}(k\circ f)-m_x(k\circ f)|<\delta_\omega,~~\forall
f\in F_\omega, k\in L_\omega.
\end{equation*}
For any
$\sigma:[[\mathcal{H}]]\to[[d]],F\subset_f[[\mathcal{H}]]_{top},K\subset_fC(\mathcal{H}^0)$
and $\delta>0,$ choose a maximum $(\rho_2,\epsilon)$-separated
subset $Q(\sigma, F, K, \delta )\subset
Orb_\nu(\pi,\sigma,F,K,\delta,\rho)$ such that for any $d_j,$ the
map $\sigma\in$ Map$([[\mathcal{H}]],[[d_j]])\mapsto Q(\sigma,
F,K,\delta)$ is Borel, and

\begin{equation*}
\exp(1)\sum\limits_{(x_1,\cdots,x_d)\in
Q(\sigma,F,K,\delta)}\exp(\sum\limits_{i=1}^d\varphi(x_i))\geq
M_\epsilon (Orb_\nu(\pi,\sigma,F,K,\delta,\rho),\rho_2,\varphi).
\end{equation*}
By the Pigeonhole Principle for any $\omega\in\Omega$ and $j\in J$
there exists $\mu_{j,\omega}\in D(\omega)$ such that

\begin{equation*}\begin{split}
&\|\sum\limits_{(x_1,\cdots,x_d)\in\{x\in
Q(\cdot,F_\omega,K_\omega,\delta_\omega):\mu_{j,\omega}=\mu_{x,\omega}\}}\exp(\sum\limits_{i=1}^d\varphi(x_i))\|_{p,\mathbb{P}_j}\\&\geq
\frac{\|\sum\limits_{(x_1,\cdots,x_d)\in
Q(\cdot,F_\omega,K_\omega,\delta_\omega)}\exp(\sum\limits_{i=1}^d\varphi(x_i))\|_{p,\mathbb{P}_j}}{|D(\omega)|}\\&\geq
\frac{\|M_\epsilon
(Orb_\nu(\pi,\cdot,F_\omega,K_\omega,\delta_\omega,\rho),\rho_2,\varphi)\|_{p,\mathbb{P}_j}}{|D(\omega)|}.
\end{split}\end{equation*}
Then choose $\mu_\omega\in D(\omega)$ so that if $J'=\{j\in
J:\mu_{j,\omega}=\mu_\omega\}$ then either $J'\in\beta$ (if $\beta$
is an ultrafilter on $J$) or $J'$ is cofinal. If $\beta=+,$ we also
require that

\begin{equation*}\begin{split}
&\limsup\limits_{j\in
J}d_j^{-1}\log\|M_\epsilon(Orb_{\mu_{j,\omega}}(\pi,\cdot,F_\omega,L_\omega,\delta_\omega),\rho_2,\varphi)\|_{p,\mathbb{P}_j}\\&=
\limsup\limits_{j\in
J'}d_j^{-1}\log\|M_\epsilon(Orb_{\mu_{j,\omega}}(\pi,\cdot,F_\omega,L_\omega,\delta_\omega),\rho_2,\varphi)\|_{p,\mathbb{P}_j}.
\end{split}\end{equation*}
This is possible due to the finiteness of $D(\omega).$ The measures
$\{\mu_\omega:\omega\in \Omega\}$ is desired. First, for any $j\in
J,\omega\in \Omega,$ if $x\in
Orb_\nu(\pi,\sigma,F_\omega,K_\omega,\delta_\omega,\rho)$ and
$\mu_{x,\omega}=\mu_{j,\omega},$ then

\begin{equation*}
|m_x(k)-\mu_{j,\omega}(k)|<\delta_\omega,~~\forall k\in L_\omega
\end{equation*}
implies $x\in
Orb_{\mu_{j,\omega}}(\pi,\sigma,F_\omega,L_\omega,\delta_\omega,\rho).$
Therefore,
\begin{equation*}\begin{split}
&\sum\limits_{(x_1,\cdots,x_d)\in\{x\in
Q(\sigma,F_\omega,K_\omega,\delta_\omega):\mu_{j,\omega}=\mu_{x,\omega}\}}\exp(\sum\limits_{i=1}^d\varphi(x_i))\\&\leq
\sum\limits_{(x_1,\cdots,x_d)\in\{x\in
Q(\sigma,F_\omega,K_\omega,\delta_\omega):\mu_{j,\omega}=\mu_{x,\omega}\}}\exp(d\mu_{j,\omega}(\varphi)+d\delta_\omega)\\&=
|\{x\in
Q(\sigma,F_\omega,K_\omega,\delta_\omega):\mu_{j,\omega}=\mu_{x,\omega}\}|\cdot\exp(d\mu_{j,\omega}(\varphi)+d\delta_\omega)\\&\leq
N_\epsilon(Orb_{\mu_{j,\omega}}(\pi,\sigma,F_\omega,L_\omega,\delta_\omega,\rho),\rho_2)\cdot\exp(d\mu_{j,\omega}(\varphi)+d\delta_\omega).
\end{split}\end{equation*}
By choice of $\mu_{j,\omega},$ this implies

\begin{equation*}\begin{split}
&\|N_\epsilon(Orb_{\mu_{j,\omega}}(\pi,\cdot,F_\omega,L_\omega,\delta_\omega,\rho),\rho_2)\|_{p,\mathbb{P}_j}\cdot\exp(d\mu_{j,\omega}(\varphi)+d\delta_\omega)\\&\geq
\frac{\|M_\epsilon
(Orb_\nu(\pi,\cdot,F_\omega,K_\omega,\delta_\omega,\rho),\rho_2,\varphi)\|_{p,\mathbb{P}_j}}{|D(\omega)|}\exp(-1).
\end{split}\end{equation*}
Due to the choice of $\mu_\omega,$ it follows that

\begin{equation*}\begin{split}
&h^\epsilon_{\mathbb{P},\mu_\omega}(\pi,\rho,\omega,2)\\&\geq
\lim\limits_{j\to\beta}d_j^{-1}\log\|N_\epsilon(Orb_{\mu_{j,\omega}}(\pi,\cdot,F_\omega,L_\omega,\delta_\omega,\rho),\rho_2)\|_{p,\mathbb{P}_j}\\&
\geq
\lim\limits_{j\to\beta}d_j^{-1}\log\|M_\epsilon(Orb_{\nu}(\pi,\cdot,F_\omega,L_\omega,\delta_\omega,\rho),\rho_2,\varphi)\|_{p,\mathbb{P}_j}
-\mu_\omega(\varphi) -\delta_\omega \\&\geq
P^\epsilon_\mathbb{P}(\pi,\rho,2,\varphi)-\mu_\omega(\varphi)
-\delta_\omega.
\end{split}\end{equation*}
$\beta\neq-$ is used in the first inequality. The finishes the proof
of the first item of Claim.

The proof of the following two items can be seen in \cite{Bow3}. We
provide it here for completeness. To prove the second item, let
$k\in C(\mathcal{G}^0),f\in[[\mathcal{H}]]_{top}$ and $\eta>0$ be a
constant. Since $k$ is continuous, there exists a constant
$\delta>0$ such that if $x,y\in\mathcal{G}^0$ with
$\rho(x,y)<\delta,$ then $|k(x)-k(y)|<\eta.$

Let $\omega\in\Omega$ be such that $k\in L_\omega,
f,f^{-1},\tau(f)\in F_\omega$ and $\delta_\omega$ is small enough so
that $\delta_\omega^2/\delta^2<\eta.$ Due to the choice of
$\mu_\omega,$ there exist a $(F_\omega,\delta_\omega^2/(100{\rm
diam}(\rho)^2))$-multiplicative $\sigma:[[\mathcal{H}]]\to[[d]]$ and
$x\in Orb_\nu(\pi,\sigma,F_\omega,K_\omega,\delta_\omega,\rho)$ such
that $\mu_{x,\omega}=\mu_\omega.$ It follows from a straightforward
computation that $|\mu_\omega(k\circ
f)-\mu_\omega(k\circ\tau(f))|\leq 2\delta_\omega+|m_x(k\circ
f)-m_{x\circ\sigma(f)}(k)|+|m_{x\circ\sigma(f)}(k)-m_x(k\circ\tau(f))|.$
It follows from $x\in
Orb_\nu(\pi,\sigma,F_\omega,K_\omega,\delta_\omega,\rho)$ that

\begin{equation*}\begin{split}
&\delta^2_\omega\geq\rho_2(f\cdot x,x\circ\sigma(f))^2\\&\geq
d^{-1}|\{1\leq i\leq d:\rho(f\cdot x_i, x_{\sigma(f)i})\geq
\delta\}|\delta^2.
\end{split}\end{equation*}
So, $\eta>\frac{\delta^2_\omega}{\delta^2}\geq d^{-1}|\{1\leq i\leq
d:\rho(f\cdot x_i, x_{\sigma(f)i})\geq \delta\}|$ which implies

\begin{equation*}\begin{split}
|m_x(k\circ f)-m_{x\circ\sigma(f)}(k)|\leq
d^{-1}\sum\limits_{i=1}^d|k(f\cdot x_i)-k(x_{\sigma(f)i})|\leq
2\eta+2\eta\|k\|.
\end{split}\end{equation*}
It follows from  $x\in
Orb_\nu(\pi,\sigma,F_\omega,K_\omega,\delta_\omega,\rho)$ that

\begin{equation*}\begin{split}
\delta_\omega^2&>\rho_2(x\circ\sigma(\tau(f)),\tau(f)\cdot
x)^2=d^{-1}\sum\limits_{i=1}^d\rho(x_{\sigma(\tau(f))_i},\tau(f)\cdot
x_i)^2\\&\geq d^{-1}{\rm diam}
(\rho)^2|\tau(\sigma(\tau(f)))\Delta\{i:x_i\in\tau(f)\}|\\&\geq
d^{-1}{\rm diam}
(\rho)^2|\tau(\sigma(f))\Delta\{i:x_i\in\tau(f)\}|-\delta_\omega^2.
\end{split}\end{equation*}
This implies that

\begin{equation*}\begin{split}
&|m_{x\circ\sigma(f)}(k)-m_x(k\circ\tau(f))|\\&=
d^{-1}\left|\sum\limits_{i\in\varsigma(\sigma(f))}k(x_{\sigma(f)i})-\sum\limits_{i:x_i\in\tau(f)}k(x_i)\right|\\&
=d^{-1}\left|\sum\limits_{i\in\tau(\sigma(f))}k(x_{\sigma(f)i})-\sum\limits_{i:x_i\in\tau(f)}k(x_i)\right|\\&\leq
d^{-1}\|k\|\cdot|\tau( \sigma(f))\Delta\{i:x_i\in\tau(f)\}|\\&\leq
d^{-1}\|k\|\frac{2\delta_\omega^2}{d^{-1}{\rm~
diam~}(\rho)^2}=\frac{2\delta^2_\omega\|k\|}{{\rm diam}~(\rho)^2}.
\end{split}\end{equation*}
Hence,

\begin{equation*}
|\mu_\omega(k\circ f)-\mu_\omega(k\circ\tau(f))|\leq
2\delta_\omega+2\eta+2\eta\|k\|+ \frac{2\delta^2_\omega\|k\|}{{\rm
diam}~(\rho)^2}.
\end{equation*}
Due to the arbitrariness of $\eta,f,k,$ we obtain

\begin{equation*}
\lim\limits_{\omega\to\Omega} |\mu_\omega(k\circ
f)-\mu_\omega(k\circ\tau(f))|=0,
\end{equation*}
for any $f\in[[\mathcal{H}]]_{top},k\in C(\mathcal{G}^0).$

It's turn to show the proof of the third of Claim. Let $k\in
C(\mathcal{H}^0).$ let $\omega\in\Omega$ be such that $k\in
K_\omega.$ Due to the choice of $\mu_\omega,$ there exists
$\sigma:[[\mathcal{H}]]\to[[d]]$ and $x\in
Orb_\nu(\pi,\sigma,F_\omega,K_\omega,\delta_\omega,\rho)$ with
$\mu_{x,\omega}=\mu_\omega.$ This implies that

\begin{equation*}
|\mu_\omega(k\circ\pi)-\nu(k)|\leq|\mu_\omega(k\circ\pi)-m_x(k\circ
\pi)|+|m_x(k\circ \pi)-\nu(k)|<2\delta_\omega.
\end{equation*}
Therefore,
$\lim\limits_{\omega\to\Omega}|\mu_\omega(k\circ\pi)-\nu(k)|=0$ as
required. The proof of Claim is completed.
\end{proof}
Let $\mu$ be a weak* accumulation point of $\{\mu_\omega:\omega\in
\Omega\}.$ It follows from the third item of Claim that
$\pi_*\mu=\nu.$ Due to (ii) and Proposition \ref{prop5.2}, we have
$\mu(k\circ f)=\mu(k\circ\tau(f))$ for any
$f\in[[\mathcal{H}]]_{top}, k\in\mathcal{G}^0.$ By Proposition
\ref{prop5.1}, $(\mathcal{G},\mu)$ is pmp. Let
$F\subset_f[[\mathcal{H}]]_{top},\varphi\in L\subset_f
C(\mathcal{G}^0)$ and $\delta>0.$ Choose $\omega\in \Omega$ to
satisfy

(1) $|\mu(k)-\mu_\omega(k)|\leq\delta/2,\forall k\in L;$

(2) $F\subset F_\omega,L\subset L_\omega,\delta_\omega\leq\delta/2.$

Then for any $\sigma:[[\mathcal{H}]]\to[[d]],x\in
Orb_{\mu_\omega}(\pi,\sigma,F_\omega,L_\omega,\delta_\omega,\rho)$
and $k\in L,$

\begin{equation*}
\left|\frac{1}{d}\sum\limits_{i=1}^dk(x_i)-\mu(k)\right|\leq\left|\frac{1}{d}\sum\limits_{i=1}^dk(x_i)-\mu_\omega(k)\right|
+|\mu_\omega(k)-\mu(k)|<\delta,
\end{equation*}
which implies that

\begin{equation*}
Orb_{\mu_\omega}(\pi,\sigma,F_\omega,L_\omega,\delta_\omega,\rho)\subset
Orb_{\mu}(\pi,\sigma,F,L,\delta,\rho).
\end{equation*}
So,

\begin{equation*}\begin{split}
&h^\epsilon_{\mathbb{P},\mu}(\pi,\rho,F,L,\delta,2)+\mu(\varphi)+\delta\\&\geq
h^\epsilon_{\mathbb{P},\mu_\omega}(\pi,\rho,F_\omega,L_\omega,\delta_\omega,2)+\mu_\omega(\varphi)\\&\geq
P^\epsilon_\mathbb{P}(\pi,\rho,2,\varphi)-\delta_\omega\\& \geq
P^\epsilon_\mathbb{P}(\pi,\rho,2,\varphi)-\delta/2.
\end{split}\end{equation*}
By taking the infimum over $F,L,\delta,$ we get

\begin{equation*}\begin{split}
&h_{\mathbb{P},\mu}(\pi,\rho,2)+\mu(\varphi)\geq
h^\epsilon_{\mathbb{P},\mu}(\pi,\rho,2)+\mu(\varphi)\\&\geq
P^\epsilon_\mathbb{P}(\pi,\rho,2,\varphi) \geq
P_\mathbb{P}(\pi,\rho,2,\varphi)-\kappa.
\end{split}\end{equation*}
Step 1 follows due to the arbitrariness of $\kappa.$

{\bf Step 2:} This step proves
$P_\mathbb{P}(\pi,\varphi)\geq\max\limits_{\mu}h_{\mathbb{P},\mu}(\pi)+\int_{\mathcal{G}^0}\varphi
d\mu.$

Fix $\mu\in M(\mathcal{G}^0)$ with $\pi_*\mu=\nu.$ Let
$\delta>0,F\subset_f[[\mathcal{H}]]_{top}$ and $K\subset_f
C(\mathcal{H}^0).$ Let $\{\varphi\}\cup\{k\circ\pi:k\in K\}\subset
K'\subset_f C(\mathcal{G}^0).$ Then for any
$\sigma:[[\mathcal{H}]]\to[[d]],$

\begin{equation*}
Orb_\mu(\pi,\sigma,F,K',\delta,\rho)\subset
Orb_\nu(\pi,\sigma,F,K,\delta,\rho).
\end{equation*}
Then

\begin{equation*}\begin{split}
&M_\epsilon(Orb_\nu(\pi,\sigma,F,K,\delta,\rho),\rho_\infty,\varphi)\\&\geq
N_\epsilon(Orb_\mu(\pi,\sigma,F,K',\delta,\rho),\rho_\infty)\cdot\exp(d\mu(\varphi)-d\delta-1).
\end{split}\end{equation*}

This implies that

\begin{equation*}\begin{split}
&\lim\limits_{j\to\beta}\frac{1}{d_j}\log\|M_\epsilon(Orb_\nu(\pi,\cdot,F,K,\delta,\rho),\rho_\infty,\varphi)\|_{p,\mathbb{P}_j}\\&\geq
\lim\limits_{j\to\beta}\frac{1}{d_j}\log
\|N_\epsilon(Orb_\mu(\pi,\sigma,F,K',\delta,\rho),\rho_\infty)\|_{p,\mathbb{P}_j}+\mu(\varphi)-\delta.
\end{split}\end{equation*}
By taking the infimum over $F, K',K,\delta$ and the supremum over
$\epsilon,$ then $P_\mathbb{P}(\pi,\varphi)\geq
h_{\mathbb{P},\mu}(\pi)+\int_{\mathcal{G}^0}\varphi d\mu.$ This
finishes the proof.
\end{proof}
Some properties of topological pressure are presented as below
without proof. It is easier to investigate these properties of
topological pressure by variational principle than by  definition.

\begin{prop}
The conditions are given as the conditions of Theorem \ref{thm5.1}.
\begin{description}
  \item[(i)] $P_\mathbb{P}(\pi,\mathbf{0})$ is topological entropy
defined in {\rm \cite{Bow3}}.
  \item[(ii)]
$P_\mathbb{P}(\pi,\varphi+c)=P_\mathbb{P}(\pi,\varphi)+c,$ where $c$
is a constant.
  \item[(iii)] $P_\mathbb{P}(\pi,\varphi+\psi)\leq
P_\mathbb{P}(\pi,\varphi)+P_\mathbb{P}(\pi,\psi),$ where $\psi\in
C(\mathcal{G}^0,\mathbb{R}).$
\item[(iv)] $\varphi\leq\psi$ implies $P_\mathbb{P}(\pi,\varphi)\leq
P_\mathbb{P}(\pi,\psi).$ In particular,
$P_\mathbb{P}(\pi,\mathbf{0})+\min \varphi\leq
P_\mathbb{P}(\pi,\varphi)\leq
P_\mathbb{P}(\pi,\mathbf{0})+\max\varphi.$
\item[(v)]  $P_\mathbb{P}(\pi,\cdot)$ is either finite valued or
constantly $\pm\infty.$
\item[(vi)] If $P_\mathbb{P}(\pi,\cdot)\neq\pm\infty,$ then $|P_\mathbb{P}(\pi,\varphi)-P_\mathbb{P}(\pi,\psi)|\leq\|\varphi-\psi\|.$
\item[(vii)] If $P_\mathbb{P}(\pi,\cdot)\neq\pm\infty,$ then
$P_\mathbb{P}(\pi,\cdot)$ is convex.
\item[(viii)] $P_\mathbb{P}(\pi,c\varphi)\leq c\cdot
P_\mathbb{P}(\pi,\varphi),$ if $c\geq1$ and
$P_\mathbb{P}(\pi,c\varphi)\geq c\cdot P_\mathbb{P}(\pi,\varphi),$
if $c\leq1.$
\item[(ix)] $|P_\mathbb{P}(\pi,\varphi)|\leq P_\mathbb{P}(\pi,|\varphi|).$
\item[(x)] $P_\mathbb{P}(\pi,\psi+\varphi\circ
f-\varphi)=P_\mathbb{P}(\pi,\psi),$ for any
$f\in[[\mathcal{H}]]_{top}.$
\end{description}
\end{prop}


\noindent {\bf Acknowledgements.} The work was supported by the
National Natural Science Foundation of China (grant No. 11271191)
and National Basic Research Program of China (grant No.
2013CB834100) and the Foundation for Innovative program of Jiangsu
province (grant No. CXZZ12 0380).

\end{document}